\documentclass[12pt]{article}

\usepackage{amsmath}
\usepackage{appendix}
\usepackage[T1]{fontenc}
\usepackage[utf8]{inputenc}
\usepackage{bookman}
\date{\today}
\usepackage{float}
\usepackage[left=1in,right=1in,top=1in,bottom=1in]{geometry}
\usepackage{amsthm}
\usepackage{amscd,amssymb}
\usepackage[all]{xy}
\usepackage{color, xcolor}
\usepackage{appendix} 
\usepackage{verbatim}
\usepackage{hyperref}
\usepackage{pdfsync}
\usepackage{mathrsfs}
\usepackage{tikz, wrapfig,array}
\usetikzlibrary{calc,patterns,
	decorations.pathmorphing,cd,
	decorations.markings, shapes.geometric, graphs, graphs.standard, quotes,shapes,chains,scopes,positioning,arrows}
\usepackage[normalem]{ulem}
\usepackage{cleveref}

\newcommand{\po}{\ar@{}[dr]|{\text{\pigpenfont R}}} \newcommand{\pb}{\ar@{}[dr]|{\text{\pigpenfont J}}}

\newtheorem{thm}{Theorem}[section]
\newtheorem{lemma}[thm]{Lemma}
\newtheorem{cor}[thm]{Corollary}
\newtheorem{prop}[thm]{Proposition}

\theoremstyle{definition}

\newtheorem{defn}[thm]{Definition}

\newcommand{\N}{{\rm N}}

\def\cG{\mathcal{G}}
\def\cC{\mathcal{C}}

\def\cW{\mathcal{W}}
\def\cC{\mathcal{C}}

\def\Wn{\mathcal{W}_\text{new}}
\def\W{\mathcal{W}}
\def\Wx{\mathcal{W}_{\times}}

\def\x{\times}

\def\fo{\text{fo}}

\title{Cofibration category structures on the category of  graphs}
	\date{}

\author{Shuchita Goyal \\
	\small Department of Mathematics\\[-0.8ex]
	\small Krea University\\[-0.8ex] 
	\small Sri City, India\\
	\small\tt shuckriya.goyal@gmail.com\\
	\and 
	Rekha Santhanam {\thanks{Corresponding author.}}\\
	\small Department of Mathematics\\[-0.8ex]
	\small Indian Institute of Technology Bombay\\[-0.8ex] 
	\small Mumbai, India\\
	\small\tt  reksan@iitb.ac.in}

\begin{document}	
	\maketitle
	\begin{abstract}
		In this article, we show that there is no cofibration category structure on the category of finite graphs with $\times$-homotopy equivalences as the class of weak equivalences. 
        Further, we show that it is not possible to enlarge the class of weak equivalences to get cofibration category structure on the category of finite graphs without including morphisms where domain and codomain have non-isomorphic stiff subgraphs.
        
		

	\end{abstract}
	
	\maketitle
	
	\section{Introduction}
	Let $\cG$ denote the category of finite undirected graphs without multiple edges, with graph morphisms being functions on vertices which preserve edges. 
    The category $\cG$ with $\x$-homotopy equivalences forms a category with weak equivalences, and we denote it by $(\cG, \Wx)$.
    The notion of  $\x$-homotopy equivalences was defined by Dochtermann in \cite{x-htpy}, and  Chih and Scull  \cite{laura-scull}, gave a  construction of the homotopy category of graphs with respect to $\x$-homotopy equivalences.
	
	   Let $\Wx$ denote the class of $\x$-homotopy equivalences in $\cG$. In \cite{lack-model-structure-RS}, we showed that there is no model structure on $\cG$ with $\Wx$ as the class of weak equivalences, irrespective of the choice for the class of cofibrations and fibrations. In the proof we used the fact that in a model category, acyclic cofibrations have left lifting property with respect to fibrations. 
	
	The advantage of defining a  model structure on $(\cG,\Wx)$ is that we can construct homotopy pushouts in $\cG$. This would potentially simplify computing the neighbourhood/hom complexes of graphs since they take double mapping cylinders of graphs to homotopy pushouts in spaces \cite{box-complex-matsushita}. 
 
 A cofibration category is a category with distinguished classes of weak equivalences and cofibrations satisfying compatibility conditions (cf. \Cref{def:cof cat}). This is a weaker notion than that of a  model structure, but several of the homotopy theoretic constructions can still be defined. In particular, we can define homotopy pushouts in a cofibration category. In a  recent article, Carranza, Doherty, Kapulkin, Opte, Saruzola and Wong \cite{CofCat-digraphs} gave a construction of a cofibration category structure on the category of digraphs.
 
 In this article, we prove that there is no choice of the class of cofibrations for which  $(\cG, \Wx)$ is a cofibration category 
 (cf. \Cref{thm:we=x-htpy+cofb=subclass-of-induced-incl} and \Cref{thm:general-we=x-htpy}). 
The class of $\x$-homotopy equivalences classify graphs up to their stiff subgraphs. We then enlarge $\Wx$, the class of weak equivalences in $\cG$ to  $\W$ (cf. \Cref{def-new we}), a class of maps that preserve all the stiff subgraphs up to isomorphism. We show that this class $\W$  does not satisfy the 2 out of 6 property (cf. \Cref{prop:Wn not satisfying 2 out of 3}). 
  Moreover, extending this class of weak equivalences $\W$ to $\Wn$ that satisfies the $2$ out of $6$ property will include morphisms between graphs having non-isomorphic stiff subgraphs (cf. \Cref{thm:Endresult}). 
    This then implies that there is no cofibration structure on finite graphs with weak equivalences classifying stiff subgraphs.


 \section{Preliminaries}
    We consider the category of finite graphs $\cG$, namely, finite undirected graphs without multiple edges with morphisms as vertex set maps preserving edges. 
	An edge in a graph with the same endpoints is called a {\it loop,} and the endpoint vertex is often referred to as a {\it looped vertex}.
	A graph without multiple edges and loops is called a {\it simple graph}.

	\begin{defn}\label{defn-fold,stiff,contractible}
		A vertex $v \in G$ is said to {\it fold} to a vertex $v' \in G$, if every neighbour of $v$ is also a neighbour of $v'$, {\it i.e.}, if the set of neighbours of $v$, $\N(v) \subseteq \N(v')$. 
		We note that the map $f : G \to G-v$ that maps each vertex (other than $v$) of $G$ to itself, and $v$ to $v'$ is a graph map. 
		In such a case, we call $f$ a {\it fold map} that folds  $G$ to $G-v$. 
		
		A graph is {\it stiff} if there is no fold in that graph. 
		
	\end{defn}

	\begin{defn}
		Let $G , H \in \cG$ be two graphs. The {\it (categorical) product} of $G$ and $H$ is defined to be the graph $G \x H$ whose vertex set is the Cartesian product, $V(G) \x V(H)$, and $(g,h)(g',h') \in E(G \x H)$ whenever $gg' \in E(G)$ and $ hh' \in E(H) $.
	\end{defn}
	
	For $n \in \mathbb{N}$, let $I_n$ be the graph with vertex set $\{0,1,\dots,n\}$ and edge set $\{ij : |i-j| \le 1\}$. We note that the graph  $I_n$ folds down to a single looped vertex for any, $n \in \mathbb{N}$ and is said to be {\it contractible}.

	\begin{defn}\cite{x-htpy}
		Two graph maps $f,g : A \to B$ are said to be {\it $\x$-homotopic}, if for some $n \in \mathbb{N}$, there exists a graph map $F : A \x I_n \to B$ such that $F|_{A\x \{0\}} = f$ and $F|_{A\x \{n\}} = g$. Two $\x$-homotopic maps are denoted as $f \simeq_{\x} g$. A graph map $f : A \to B$ is a {\it $\x$-homotopy equivalence}, if there exists a graph map $g: B \to A$ such that $gf \simeq_{\x} 1_A$ and $fg \simeq_{\x} 1_B$, where $1_X$ denotes the identity map on $X$.
	\end{defn}

	By a {\it stiff subgraph} of a graph $G$,  we always mean a subgraph of $G$, which is stiff, that $G$ folds down to and denote it by $G_s$. We note that a graph $G$, via different sequences of folds, may fold down to different subgraphs of it, but all of these subgraphs are isomorphic \cite[Proposition 6.6]{x-htpy}.
	
	\begin{defn}\cite{cofibration-category-defn}\label{def:cof cat}
		A category with weak equivalences and cofibrations as distinguished subclasses of morphisms is called a {\it cofibration category} if the following axioms are satisfied:
		\begin{enumerate}
			\item Weak equivalences satisfy ‘2 out of 6’ property, {\it i.e.}, for any three composable morphisms $f: A \to B,\  g : B \to C, \ h: C \to D$, if $hg,gf$ are weak equivalences, then so are $f,g,h$ and $hgf.$
			\item Every isomorphism of the category is an acyclic cofibration, {\it i.e.} a map which is a weak equivalence and a cofibration.
			\item An initial object exists in the category.
			\item Every object X of the category is cofibrant, {\it i.e.}, the map from the initial object to $X$ is a cofibration.
			\item Cofibrations and acyclic cofibrations are preserved under cobase change.
			\item Every morphism in the category factors as a composite of a cofibration followed by a weak equivalence.
		\end{enumerate}
	\end{defn} 

	\section{Non existence of cofibration category structure on \texorpdfstring{$(\cG,\Wx)$}{(G,wx)} }\label{sec:new cofib class}
	
	We now consider the category of graphs, $\cG$, with $\x$-homotopy equivalences. An empty graph is a graph whose vertex set and edge set are empty. One can easily verify that an empty graph is an initial object in $\cG$, and hence the axiom (3) of \Cref{def:cof cat} is satisfied. Recall that $\Wx$ denotes the class of $\x$-homotopy equivalences.

	\begin{lemma}\label{2 out of 6 ppty proof}
		The class  $\Wx$ satisfies the `$2$ out of $6$' property.
	\end{lemma}
	
	\begin{proof}
		Consider the following diagram in $\cG$ 
		$$\xymatrix{
			A\ar[r]^f & B \ar[r]^g & C\ar[r]^h & D\\
		}$$
		where $gf,\ hg \in \Wx$. We show that all of $f,g,h,hgf$ belongs to $\Wx$. Since $gf \in \Wx$, there exists graph maps $\alpha : C \rightarrow A$, $F : C \x I_{n_1} \to C$ and $ F': A \x I_{n_2} \to A $ for some natural numbers $n_1,n_2 $ such that $F : gf \alpha \simeq_{\times} 1_C$, $F' : \alpha gf \simeq_{\times} 1_A$. Similarly for $hg$, there exists $\beta : D \rightarrow B$, $G : D \x I_{n_3} \to D$ and $G': B \x I_{n_4} \to B$ for some naturals $n_3, n_4$ such that $G : hg \beta \simeq_{\times} 1_D$, $G' : \beta hg \simeq_{\times} 1_B$.
		
		Now consider the map $\alpha g \beta : D \rightarrow A$. Using associativity of graph maps, $	(\alpha g \beta) (hgf)   =\alpha  g (\beta h g) f$. Since $G' :  \beta hg \simeq_{\times} 1_B$, $\alpha g G' f: \alpha  g (\beta h g) f \simeq_{\x}  \alpha g 1_B f =  \alpha gf \simeq_{\x} 1_{A}  \ ( via\ F' )$. 
		Again using the associativity, we get $(hgf)(\alpha g \beta)    = h (gf \alpha) g \beta $.  Since $F : gf \alpha \simeq_{\times} 1_C$, $ h (gf \alpha) g \beta \simeq_{\times} h 1_C g \beta = h g \beta   \simeq_{\times} 1_D \ ( via\ G)$.
		Therefore, $hgf$ is a $\x$-homotopy  equivalence and thus $hgf \in \Wx$. 
		
		Next, we note that the composition of two $\x$-homotopy equivalences is again a $\x$-homotopy equivalence and a map {\it i.e.} $\times$-homotopic to a weak equivalence is itself a weak equivalence.
		Consider $\beta : D\rightarrow B$. Since the map $hg$ is a $\times$-homotopy equivalence and $G : hg \beta \simeq_{\x} 1_{D}$, therefore $\beta$ is also a $\times$-homotopy equivalence, {\it i.e.} $\beta \in \Wx$. So, $\beta (hgf) = (\beta hg) f \simeq_{\times} f$ implying that $f \in \Wx$. A similar argument shows that $h $ is also a $\x$-homotopy equivalence. Finally, observe that $\beta h  : C \to B$ is a left $\x$-homotopy inverse of $g$ while $f \alpha : C \to B$ is a right inverse. Therefore, $\beta h \simeq_{\times} (\beta h) g (f \alpha)$ and $(\beta h) g (\beta h) \simeq_{\times} f \alpha$ implies that $ g (\beta h) \simeq_{\times} g (\beta h) g (f \alpha) \simeq_{\times} g (f \alpha) \simeq_{\times} 1_{C}$, and hence $g$ is also a weak equivalence.
		
		Therefore $\Wx$ satisfies `2 out of 6' property.
	\end{proof}

        We now prove that if there is a cofibration category structure on $(\cG, \Wx)$, then the class of cofibrations is a subclass of inclusions.

    \begin{prop} \label{lem:cofib be inclusions}
		Let $f: A \to B$ be a graph map which is a $\x$-homotopy equivalence and not injective. 
		Then there exists a graph $C$ and a map $g: A \to C$ such that the cobase change map from $C \to P$ is not a $\x$-homotopy equivalence, where $P$ is the pushout of the diagram $B  \xleftarrow{f} A \xrightarrow{g} C$. 
	\end{prop}
	\begin{proof}
		Let $f(a_1)=f(a_2)$ for two distinct vertices $a_1, a_2 \in V(A)$.
		
		If $A$ is simple, then $B$ is  simple, as $f$ is a $\x$-homotopy equivalence. 
		Note that $f(a_1)$ is not a looped vertex, therefore $a_1 a_2 \notin E(A)$. 
		Let $C$ be the graph with $V(C)=V(A)$ and $E(C)=E(A) \cup \{a_1 a_2\}$. 
		Let $g: A \to C$ be the graph inclusion. 
		Then $C$ is simple, but the pushout of the diagram $B  \xleftarrow{f} A \xrightarrow{g} C$ is not simple and hence not $\x$-homotopy equivalent to $C$.
		
		If $A$ is not simple, but neither of $a_1$ or $a_2$ is looped, then we construct $C$ as $C=(A\sqcup C_7)/ a_1 \sim 1, $ where $1$ is a vertex in $C_7$ with an edge added between $a_1, a_2$. Then the pushout has a $C_7$ with a looped vertex as a stiff subgraph. Any other copy of such a graph in $C$ has to arise from one in $A$ and therefore $B$. This implies $C$ is not $\x$-homotopy equivalent to the pushout. 
		
		If either $a_1$ or $a_2$ in $A$ is looped, then we can add a loop at $1$ or $a_2$ in $C$ appropriately to repeat the previous argument. 
	\end{proof}
	
	    \begin{cor}\label{thm:general-we=x-htpy}
	If there exists a cofibration structure on $(\cG,\Wx)$ then cofibrations have to be a subclass of induced inclusions. 
 \end{cor}
 \begin{proof} 
    By \Cref{lem:cofib be inclusions} we know that cofibrations have to be a subclass of inclusions. 
	In \cite[Remark 3.2]{lack-model-structure-RS}, we have shown that if for any  $f : A \to B$ which is a   $\x$-homotopy equivalence and a non-induced graph inclusion, there is a graph $C$ and a graph map $g : A \to C$ such that the cobase change of $f$ along $g$ is not a $\x$-homotopy equivalence. 

     Thus any class of cofibrations on $(\cG,\x)$, if it exists, has to be a subclass of induced inclusions. 
    \end{proof}



     If the class of cofibrations is induced inclusions, then for any graph map $f: A \to B$, we can construct a new graph $M_f$ such that $A \subseteq M_f$ is an induced subgraph and $M_f \simeq_{\x} B$. Then this construction gives a factorisation of any graph map $f: A \to B$ as a composition of an induced inclusion followed by a weak equivalence. We give this construction next.
	
    \begin{defn}
		Let $A, G \in \cG$ be graphs, and $f : A \to G$ be a graph map. The \textit{quotient} of $G$ by $f(A)$, denoted by $G / f(A)$, is the graph defined as, 
        \begin{align*}
            V(G/f(A))=&   V(G)/ V(f(A)), \\
            E(G/f(A)) = &   \{ \{[x],[y]\} \subseteq  V(G)/ V(f(A)) \ | \ xy \in E(G)\}.
        \end{align*}
            
	\end{defn}
	
	Let $f:A\to B$ be a graph map. We define $M_f$ to be the quotient graph such that $V(M_f) = V\Big ((A\x I_1 )\bigsqcup\limits_i B \Big ) = \Big (V(A\x I_{1})\cup V(B) \Big )/f(a)\sim (a,0)$ and
	$[x][y]\in E(M_f)$ if and only if there exist $x'\in [x], y'\in [y]$ such that
	$x'y'\in E(A\x I_n)\cup E(B)$.
	
	It is easy to see that $A$ is an induced subgraph of $M_f$. Furthermore, for every $a \in V(A)$, the vertex $[f(a)]$ folds to the vertex $(a,1) \in V(M_f)$ and hence $M_f \simeq_{\x} B$.

 	However we have shown that if the class of weak equivalences are the $\x$-homotopy equivalences \cite[Lemma 3.2 and Proposition 3.1]{lack-model-structure-RS},  then cofibrations need to satisfy further properties, which are stated below.

    \begin{defn}\label{quasicof}
        An induced inclusion $A \subseteq B$ of graphs is a {\it quasi-cofibration} if $B$ can be folded to its stiff subgraph via a composition of relative and restricted folds in the codomain with respect to the domain graph, and no other folds are possible at any stage.
    \end{defn}
	
	Let the class of quasi-cofibrations be denoted by $\cC$.
	The fact that $\cC$ satisfies axiom (2) and (4) of \Cref{def:cof cat} is then immediate from the definition.
	
	By the definition of quasi-cofibrations and the discussion from \cite[Section 3]{lack-model-structure-RS}, $\cC$ satisfies axiom (5) of \Cref{def:cof cat}. We now show that $\cC$, however, will not satisfy axiom (6).

    \begin{thm}\label{thm:we=x-htpy+cofb=subclass-of-induced-incl}
There is no cofibration category structure on $\cG$ with $\x$-homotopy equivalences as weak equivalences. 
\end{thm}	
	
	\begin{proof}
        We have already shown that any class of cofibrations if it exists, is a subclass of quasi-cofibrations. 
		Let $h: A \to B$ be a graph map between two stiff graphs $A$ and $B$ such that $|V(B)| < |V(A)|$. We prove this theorem by showing that maps of the form $h$ cannot be factored as a quasi-cofibration followed by a $\x$-homotopy equivalence.
  
	On the contrary, we assume that there is such a factorization $A \xrightarrow{f} B' \xrightarrow{g} B$, {\it i.e.}, $h = gf$ with $f$ a quasi-cofibration and $g$ a $\x$-homotopy equivalence. 
		
	We first note that $B' \simeq_{\x} B$ and $B$ implies that $|V(B'_s)| = |V(B)|$, where $B'_s$ denotes a stiff subgraph of $B'$. 
		Therefore $|V(B)| < |V(A)|$ implies that after a few relative folds in $B'$ with respect to $A$, there will be some restricted folds as well, which is a contradiction to the fact that $A$ is stiff.
	\end{proof}

 For example, the natural map $f: C_6 \to K_2$ cannot be factored as a quasi-cofibration followed by a weak equivalence.

	\section{  Enlarging the class of weak equivalences}\label{sec:new weak equivalences}
	
	In \cite[Theorem 5.3]{droz} Droz points out that there is a continuum of model structures on the category of graphs, and each of them will give rise to a cofibration structure. We have now shown that none of these arise with $\x$-homotopy equivalences as the weak equivalences.
	Then a natural question is whether we can enlarge the class of weak equivalences in a way that they preserve the inherent property of $\x$-homotopy equivalences, {\it i.e.}, maps having isomorphic stiff subgraphs. 
	
	Recall that any $\x$-homotopy equivalence can be written as a composition of folds, isomorphism and unfolds. To see this, let $g: A \to B$ be a $\x$-homotopy equivalence. Since any two $\x$-homotopy equivalent graphs have isomorphic stiff subgraphs, $A_s \cong B_s$ \cite[Proposition 6.6]{x-htpy}, there is an  isomorphism $g': A_s \to B_s$ such that the following diagram commutes,
	\begin{equation*}
		\xymatrix{ A \ar[d]_{\fo_A} \ar[r]^g & B \ar[d]^{\fo_B} \\ 
			A_s \ar[r]_{g'} & B_s }
	\end{equation*}
    where $\fo_A$ and $\fo_B$ denotes map that folds $A$ and $B$ to their stiff subgraphs respectively.
 We can enlarge the class $\Wx$ to contain graph maps which when restricted to a stiff subgraph of $A$  is an isomorphism onto its image. Further, we want the image of a stiff subgraph of $A$ to be isomorphic to a stiff subgraph of $B$. 
	
    \begin{defn} \label{def-new we}
	Let $\W$ be the class of maps in $\cG$ given by maps $f: A \to B$ such that, for any subgraph $T \subseteq A$, $T \simeq A_s$, $f|_{T}: T \to f(T)$ is an isomorphism and $f(T)\simeq B_s$.  
    \end{defn}
    
	Note that if $f:A \to B$ is in $\W$ and  $A$ folds down to a stiff subgraph $T$ it is possible that   $B$ does not fold down to $f(T)$.	
    However, this class does not satisfy the $2$ out of $3$ property and hence does not satisfy the $2$ out of $6$ property as we show next. 
 
    \begin{prop} \label{prop:Wn not satisfying 2 out of 3}
      The morphism class $\W$ in $\cG$ does not satisfy the $2$ out of $6$ property. 
     \end{prop}
    \begin{proof}
       Since 2 out of 6 property implies the 2 out of 3 property,
       to prove this statement, it is sufficient to give an example in $\W$ which does not satisfy $2$ out of $3$ property. 
		Let $A$ be a graph obtained from taking a wedge of $C_5$, with the graph $C_3$ along an edge (cf. \Cref{fig:counterexample to 2 out of 3 property }). Also, let $B$ and $C$ be the graphs as shown on the right side of \Cref{fig:counterexample to 2 out of 3 property }. It is easy to see that $A$ is a stiff graph, and hence it is its only stiff subgraph. Further, note that in $B$, the vertices $a,d$ folds to $x$; $c,e$ folds to $y$ and $b$ fold to $z$. Therefore, the stiff subgraph $B_s$ of $B$ is isomorphic to $A$. 
  
    \begin{figure}[H]
		\centering
		
		\begin{tikzpicture}[main_node/.style={circle,draw,minimum size=1em,inner sep=1.6pt},scale=0.7]
			
			\begin{scope}[shift={(-5,0)}]
				\node[main_node] (0) at (-4.2285714285714291, 2.785714285714286) {$x$};
				\node[main_node] (1) at (-2.828571428571429, -0.9) {$z$};
				\node[main_node] (2) at (-4.857142857142858, -0.8571428571428577) {$y$};
				\node[main_node] (3) at (-2.4, 2.4) {1};
				\node[main_node] (4) at (-1.800000000000001, 0.2428571428571434) {3};
				\node[main_node] (5) at (-1, 1.4000000000000004) {2};
				
				\path[draw, thick]
				(2) edge node {} (0) 
				(0) edge node {} (1) 
				(2) edge node {} (1) 
				(0) edge node {} (3) 
				(3) edge node {} (5) 
				(5) edge node {} (4) 
				(4) edge node {} (1) 
				;
				\node  at (-4,-3) {$A$};
			\end{scope}
			
			\begin{scope}
				\node[main_node] (0) at (-0.2285714285714291, 4.785714285714286) {$x$};
				\node[main_node] (1) at (-0.5857142857142854, 2.5142857142857142) {$a$};
				\node[main_node] (2) at (1.042857142857143, 1.8999999999999997) {$b$};
				\node[main_node] (3) at (3.828571428571429, -1) {$z$};
				\node[main_node] (4) at (-1.4857142857142858, 0.1428571428571428) {$d$};
				\node[main_node] (5) at (0.5, 0.11428571428571477) {$c$};
				\node[main_node] (6) at (-2.114285714285714, 1.9142857142857146) {$e$};
				\node[main_node] (7) at (-4.857142857142858, -0.8571428571428577) {$y$};
				\node[main_node] (8) at (2.2857142857142856, 4.642857142857142) {1};
				\node[main_node] (9) at (4.200000000000001, 0.8428571428571434) {3};
				\node[main_node] (10) at (3.885714285714286, 3.4000000000000004) {2};
				
				\path[draw, thick]
				(7) edge node {} (0) 
				(1) edge node {} (7) 
				(0) edge node {} (6) 
				(0) edge node {} (5) 
				(0) edge node {} (2) 
				(0) edge node {} (3) 
				(1) edge node {} (2) 
				(2) edge node {} (5) 
				(5) edge node {} (4) 
				(4) edge node {} (6) 
				(6) edge node {} (1) 
				(7) edge node {} (2) 
				(7) edge node {} (4) 
				(7) edge node {} (3) 
				(1) edge node {} (3) 
				(3) edge node {} (5) 
				(0) edge node {} (8) 
				(8) edge node {} (10) 
				(10) edge node {} (9) 
				(9) edge node {} (3) 
				;
				\node at (0,-3) {$B=C$};
			\end{scope}
			
		\end{tikzpicture}
		\caption{Counterexample (2 out of 3 property in $\W$)} \label{fig:counterexample to 2 out of 3 property }
		
	\end{figure}
 
  Now, define $f : A \to B$ as the inclusion map, and $g : B \to C$ be the map that sends $\{a,d,x\}$ to $x$, $\{c,e,y\}$ to $y$, $\{b,z\}$ to $z$, $1\mapsto 1, \ 2 \mapsto 2, \ 3 \mapsto 3$. By definition $f \in \W$ as $A$ itself is stiff and so is $gf$. If $T = \{a,b,c,d,e,x\} \subseteq V(B)$ and $B[T]$ is the subgraph of $B$ induced by the set $T$, then $B[T] \cong B_s$, a stiff subgraph of $B$, however, $g(B[T])$ is isomorphic to a $C_3$ which is not isomorphic to a stiff subgraph of $C$, thereby implying that $g$ does not belong to $\W$. Therefore, $\W$ does not satisfy the 2 out of 3 property.

\end{proof}

\begin{cor}
 The category $\cG$ with the class of morphisms $\W$ is not a category with weak equivalences. In particular, $\cG$ is not a cofibration category with $\cW$ as weak equivalences.
\end{cor}

However, $\W$ does have several nice properties that weak equivalences are expected to satisfy. 

\begin{lemma} \label{lem:propertiesWn}
    Let $\W$ be the class of morphisms as defined in \Cref{def-new we}. Then
    \begin{enumerate}
    
        \item  $\W$ is closed under compositions.
        \item \label{iso}  All isomorphisms, folds, unfolds are in $\W$, hence all the $\x$-homotopy equivalences are in $\W$.
        \item  Not every map in $\W$ is a $\x$-homotopy equivalence. 
        \item  Given composable graph morphisms $f$ and $g$, if  $g, gf \in \W$ then $f\in \W$.
     \end{enumerate}
\end{lemma}
\begin{proof}
\begin{enumerate}
    \item Compositions are preserved in $\W$. Let $f:A\to B$ and $g:B \to C$ be in $\W$. Then for any subgraph $T$ of $A$ isomorphic to $A_s$, $T \to f(T)$ is an isomorphism such that $f(T)$ is isomorphic to $B_s$. Then $g|_{f(T)}$ is an isomorphism onto $g(f(T))$ and $g(f(T))$ is isomorphic to $C_s$. 
			
    \item It is straightforward to check that all the isomorphisms, folds and unfolds belong to $\W$. Note that every $\x$-homotopy equivalence $A \to B$ is a composition of folds, isomorphisms and unfolds; and since all of these are in $\W$ which is closed under taking compositions, every $\x$-homotopy equivalence is in $\W$.

		\item  To show that not every weak equivalence is a $\x$-homotopy equivalence, consider the graphs $A, B$ with $V(A) = \{1,2,3,4,5\}, \ E(A) = \{12,23,13, 14, 45\}$, and $V(B) = \{a,b,c\}, \ E(B) = \{ab,bc,ac\}$ (see \Cref{fig:we - x-htpy}). 

        \begin{figure}[H]
            \centering
        \begin{tikzpicture}[main_node/.style={circle,draw,minimum size=1em,inner sep=1pt]}]

\node[main_node] (0) at (-3.10000010899135, 4.842857142857143) {1};
\node[main_node] (1) at (-3.9857142857142858, 3.2571428571428576) {2};
\node[main_node] (2) at (-2.242857251848493, 3.2571428571428576) {3};
\node[main_node] (3) at (-4.842857142857143, 2.7428571428571424) {4};
\node[main_node] (4) at (-1.8428570338657924, 2.7428571428571424) {5};
\node[main_node] (5) at (2.0999996730259483, 4.842857142857143) {$a$};
\node[main_node] (6) at (1.20714286804199215, 3.2) {$b$};
\node[main_node] (7) at (2.8285711015973764, 3.2) {$c$};

\node at (-3,2) {$A$};
\node at (2,2) {$B$};

 \path[draw, thick]
(0) edge node {} (1) 
(1) edge node {} (2) 
(2) edge node {} (0) 
(3) edge node {} (4) 
(3) edge node {} (0) 
(5) edge node {} (6) 
(6) edge node {} (7) 
(7) edge node {} (5) 
;

\end{tikzpicture}     
            \caption{Map in $\W \setminus \Wx$}
            \label{fig:we - x-htpy}
        \end{figure}

	Clearly $A_s \cong B_s \cong K_3$, a   complete graph on 3 vertices. 
        Let $f : A \to B$ be the map that sends $1 \mapsto a, \{2,4\} \mapsto b$ and $\{3,5\} \mapsto c$. 
			Then $f$ maps $A_s$ isomorphically onto $B_s = B$ but $f$ is not a $\x$-homotopy equivalence.

       \item 	Let  $f:A \to B $, $g:B \to C $ be graph maps.  Let $g$ and $gf$ be in $\W$. Let $S \subseteq A$ be a subgraph such that $S \cong A_s$. Then $gf|_{S}$ maps $S$ isomorphically to $gf(S)$, a subgraph isomorphic to the stiff subgraph of $C$. Also, $gf|_S$ being an isomorphism implies that $f|_S$ is injective and hence $f|_S : S \to f(S)$ is an isomorphism. Since $S \cong A_s \cong B_s$, $S \cong f(S)$ imply that $f(S) \cong B_s$. Thus $f \in \Wn$ .

\end{enumerate} 
\end{proof}

 We now show that we cannot extend this class  $\W$ to a class $\Wn$ which satisfies the $2$ out of $6$ property, and morphisms have isomorphic stiff subgraphs. 
    
    \begin{thm}\label{thm:Endresult}
  Let $\Wn$ be the smallest class of morphisms containing $\W$  and satisfying the $2$ out of $6$ property. Then $\Wn$ contains graph morphisms where the domain and codomain have non-isomorphic stiff subgraphs. 
    \end{thm}

    \begin{proof}
        We construct graphs $A,B,C,D$ and composable graph maps $f:A \to B$, $g: B \to C$ and $h:C \to D$ such that the domain and codomain of $f$ have non isomorphic stiff subgraphs but $gf,hg \in \Wn$.
        Consider the graph $D$ with the vertex set $V(D) = \{1,2,3,4,5\}$ and the edge set $E(D) = \{11, 12, 14, 23, 24, 25, 33, 34, 44, 55\}$. 
        Let $A,B,C$ be induced subgraphs of $D$ on vertex sets $\{1\}, \ \{1,2,3\}, \ \{1,2,3,4\}$ respectively (see \Cref{fig:unwanted maps in W}). 
        Then $A \subset B \subset C \subset D$.
        Let $f: A \to B$, $g: B \to C$ and $h: C \to D$ be the induced inclusions. Then  we note that the graphs $A$ and $B$ are stiff. Moreover, in $C$, the sequence of folding 3 to 4, then 4 to 1, and finally 2 to 1 implies that the stiff subgraph $C_s$ of $C$ is $A$. 
        Similarly the sequence of folding 3 to 4 followed by 4 to 1 in $D$ gives the stiff subgraph $D_s$ of $D$ which is isomorphic to $B$. 
        
 \begin{figure}[H]
     \centering
     \tikzstyle{black1}=[circle, draw, fill=black!100, inner sep=0pt, minimum width=4pt]
	\tikzstyle{ver}=[]
	\tikzstyle{extra}=[circle, draw, fill=black!00, inner sep=0pt, minimum width=2pt]
	\tikzstyle{edge} = [draw,thick,-]
	\tikzstyle{light} = [draw, gray,-]
		
    \begin{tikzpicture}[scale=0.5]
        \begin{scope}[shift={(0,0)}]
            \node (0) at (-3,0) {1};
            \draw[] (0) circle [radius=.5];
            \draw (0) to[out=250,in=290, distance = 2cm] (0);
            \node at (-3, -5) {$A$};
        \end{scope}
        \begin{scope}[shift={(3,0)}]
            \node (0) at (-3,0) {1};
            \node (1) at (0,3) {2};
            \node (2) at (3,0) {3};
            \draw[] 
                (0) circle [radius=.5]
                (1) circle [radius=0.5]
                (2) circle [radius=0.5];

            \draw (0) to[out=250,in=290, distance = 2cm] (0);
            \draw (2) to[out=250,in=290, distance = 2cm] (2);

            \path[edge]
                (2) edge node {} (1) 
                (0) edge node {} (1) ;
            \node at (0, -5) {$B$};
        \end{scope}

       \begin{scope}[shift={(12,0)}]
            \node (0) at (-3,0) {1};
            \node (1) at (0,3) {2};
            \node (2) at (3,0) {3};
            \node (3) at (0,-1) {4};
            \draw[] 
                (0) circle [radius=.5]
                (1) circle [radius=0.5]
                (2) circle [radius=0.5]
                (3) circle [radius=0.5];

            \draw (0) to[out=250,in=290, distance = 2cm] (0);
            \draw (2) to[out=250,in=290, distance = 2cm] (2);
            \draw (3) to[out=250,in=290, distance = 2cm] (3);

            \path[edge]
                (3) edge node {} (1) 
                (2) edge node {} (1) 
                (0) edge node {} (1) 
                (0) edge node {} (3) 
                (3) edge node {} (2);
            \node at (0, -5) {$C$};

        \end{scope}

        \begin{scope}[shift={(22,0)}]
            \node (0) at (-3,0) {1};
            \node (1) at (0,3) {2};
            \node (2) at (3,0) {3};
            \node (3) at (0,-1) {4};
            \node (4) at (-3,3) {5};
            \draw[] 
                (0) circle [radius=.5]
                (1) circle [radius=0.5]
                (2) circle [radius=0.5]
                (3) circle [radius=0.5]
                (4) circle [radius=0.5];

            \draw (0) to[out=250,in=290, distance = 2cm] (0);
            \draw (2) to[out=250,in=290, distance = 2cm] (2);
            \draw (3) to[out=250,in=290, distance = 2cm] (3);
            \draw (4) to[out=160,in=200, distance = 2cm] (4);

            \path[edge]
                (3) edge node {} (1) 
                (2) edge node {} (1) 
                (0) edge node {} (1) 
                (1) edge node {} (4) 
                (0) edge node {} (3) 
                (3) edge node {} (2);
            \node at (0, -5) {$D$};
        \end{scope}


        \end{tikzpicture}     \caption{The 2 out of 6 implying existence of some unwanted maps in $\Wn$}
        \label{fig:unwanted maps in W}
     \end{figure}

        With this, it is clear that $gf$ and $hg$ both belong to $\W$, and hence belong to $\Wn$. Thus by the 2 out of 6 property of $\Wn$, we have $f,g,h,hgf \in \Wn$. In particular, $f: A \to B $ is a weak equivalence. However, the stiff subgraphs of $A$ and $B$ are non-isomorphic.
    \end{proof}


Therefore via \Cref{thm:Endresult}, we have shown that it is not possible to extend the class of  $\x$-homotopy  equivalences in a manner that gives rise to a cofibration category where weak equivalences distinguish graphs having non-isomorphic stiff subgraphs.

	

	\end{document}